\documentclass[graybox]{svmult}

\usepackage{type1cm}        % activate if the above 3 fonts are
\usepackage{makeidx}         % allows index generation
\usepackage{graphicx}        % standard LaTeX graphics tool
\usepackage{multicol}        % used for the two-column index
\usepackage[bottom]{footmisc}% places footnotes at page bottom

\usepackage{newtxtext}       % 
\usepackage[varvw]{newtxmath}       % selects Times Roman as basic font

\makeindex             % used for the subject index
\usepackage{bbold}
\usepackage{nicefrac}
\begin{document}

\title*{Overview on uncertainty quantification in traffic models via intrusive method}
\author{Elisa Iacomini}
\institute{Elisa Iacomini \at  Institut für Geometrie und Praktische Mathematik RWTH Aachen, Templergraben 55 52056 Aachen Germany, \email{iacomini@igpm.rwth-aachen.de}
%\and Name of Second Author \at Name, Address of Institute \email{name@email.address}
}
\maketitle

\abstract*{Each chapter should be preceded by an abstract (no more than 200 words) that summarizes the content.}

\abstract{We consider traffic flow models at different scales of observation.  Starting from the well known hierarchy between microscopic, kinetic and macroscopic scales, we will investigate the propagation of uncertainties through the models using the stochastic Galerkin approach. Connections between the scales will be presented in the stochastic scenario and numerical simulations will be performed.
 }

\section{Introduction}
\label{sec:intro}
%\textcolor{blue}{In the last decades, an intensive research activity in the field of traffic flow modeling flourished due to its crucial role in the management of vehicular traffic.,}

%\textcolor{orange}{In the last decades, many traffic models have been developed and investigated resorting to different approaches}

{In the last decades, several traffic models have been developed and investigated at different spacial and temporal scales.}

Starting from the natural idea of tracking every single vehicle, several follow-the leader models grew up for computing positions, velocities and accelerations of each car by means of systems of ordinary differential equations (ODEs) \cite{bando1995dynamical,newell1961nonlinear}. Zooming out, the approaches vary from kinetic \cite{herty2020bgk,puppo2016kinetic,tosin2019kinetic}, which provides a statistical description of traffic taking into account cars-to-cars interactions and mass distribution of traffic, to macroscopic fluid-dynamics \cite{aw2000SIAP,lighthill1955RSL,Piccoli2012review}, focusing on average quantities by means of partial differential equations (PDEs), in particular conservation laws.

Hierarchies and links between the those scales have been widely studied, in particular mathematical connections between different types of models, especially microscopic and kinetic ones which converge to macroscopic models in certain suitable limit, \cite{aw2002derivation,di2017many,herty2020bgk,helbing2001TR,tosin2021boltzmann}.

The choice of the scale of observation mainly depends on the aim of the modeling, i.e. forecast traffic evolution on highways or traffic jams at a junction, on the number of involved vehicles and so on. Unfortunately, due to the particular structure of the corresponding models, they are usually studied in rather separate works.
  
Therefore, one of the main focus of this work is to provide a uniform setting to study uncertainty quantification in traffic at different scales.

%The natural way of describing traffic evolution is tracking each single vehicle, namely considering its position and velocity, and reconstructing its trajectory. This representation belongs to the so called microscopic scale, 

Indeed, in recent works it has been pointed out how traffic is exposed to the presence of various sources and types of uncertainty, also at different scales of observation \cite{tosin2021uncertainty,wegener1996kinetic}. For example, it is well known that real data may be affected by noise and errors in the measurements, or that the reaction time of drivers and cars is not a deterministic event. Therefore, investigating how to include and to model the uncertainty at different scales in a consistent way, and how it affects the existing hierarchy are the main goals of this work.

We will introduce the uncertainty in the initial data at a microscopic, mesoscopic and macroscopic scale respectively and we will analyze the obtained stochastic models. In order to study the propagation of input uncertainty through the models, several approaches have been proposed in the literature and can be classified in non-intrusive, e.g. based on sampling (Monte-Carlo)  or based on collocation \cite{S4}, and intrusive methods \cite{S21,Gottlieb2001}, where the stochastic quantities are described by a series of orthogonal functions, known as generalized polynomial chaos (gPC) expansion \cite{S2,S16, S1,S3}.  The idea is then to substitute the series in the governing equations and project, using a Galerkin projection, in order to obtain a system of deterministic coefficients. From the coefficients, stochastic moments can be recovered at each point in space and time without performing the simulation several times, as in the Monte-Carlo simulations. The detailed knowledge of the stochastic moments, allowed by the intrusive methods both in space and time, led us to follow this approach, namely the stochastic Galerkin method.

Many challenges arise here, since some desired properties of the original system are not necessarily transferred to the intrusive formulation \cite{S15}. In particular, we refer to the hyperbolicity in the macroscopic models, \cite{FettesPaper,SI2}. In fact, the deterministic Jacobian of the projected system differs from the random Jacobian of the original system. For more details we refer to \cite{gerster2021stability}. 
Therefore, applying the intrusive stochastic Galerkin method to hyperbolic equations is an active field of research ~\cite{H5,kusch2019maximum,petterson}. 

\vspace{0.3cm}

The paper is organized as follows: uncertainty and the stochastic Galerkin method are introduced in Section \ref{sec:unc}.  In Section \ref{sec:micro} stochastic microscopic models at first and second order are presented. Section \ref{sec:meso} deals with the mesoscopic scale, i.e. stochastic kinetic models are considered, in particular the BGK model. In Section \ref{sec:macro} the macroscopic models are described and recovered from the previous ones and the main properties are illustrated. Numerical tests will be provided. We conclude the work with some comments and an overview on future directions.

\section{Stochastic Galerkin approach}\label{sec:unc}

 In order to describe the presence of uncertainty, we introduce a (possibly multi-dimensional) random variable $\omega$. Let $\omega$ be defined on the probability space $(\Omega_\omega,\mathcal{F}(\Omega),\mathbb{P})$. Then we denote by $\xi = \xi(\omega):\Omega_\omega \to \Omega \subset \mathbb{R}^d$ a (possibly $d$-dimensional) real-valued random variable. Assume further that $\xi$ is absolutely continuous with respect to the Lebesgue measure on $\mathbb{R}^d$ and denote by ${p_\Xi}(\xi):\Omega \to\mathbb{R}_+$ the probability density function of $\xi$. For simplicity we assume that the uncertainty enters only in the initial data, so that it affects the initial configuration. 

To study the arising stochastic models, we briefly recall the intrusive approach we will employ, namely the stochastic Galerkin method introduced by \cite{S21}.\\
Here, a random field $u(t,\xi)$, namely the stochastic input, can be expressed  by a spectral expansion \cite{S16}
under the assumption of being sufficiently regular and in particular $L^2_p (\Omega)$
\begin{equation}\label{eq:gPC_generic}
u_i(t,\xi)= \sum_{k=0}^\infty \hat{u}_{k}(t) \phi_k(\xi) 
\end{equation} 
where  $\phi_k \in L^2(\Omega, p_\Xi)$ are basis functions, typically chosen orthonormal with respect to the weighted scalar product, and $\{ \hat{u}_{k}(t) \}_{k=0}^\infty$ is a set of coefficients:
\begin{equation}\label{eq:gPC0}
\hat{u}_{k}(t) = \int_\Omega u(t,\xi) \phi_k(\xi) p_\Xi(\xi) d\xi.
\end{equation}
The previous expansion is truncated at $K$ to obtain an approximation with $K+1$ moments.  The projection of $u(t,\cdot)$ to the span of the $K+1$ base functions is denoted by
\begin{equation}\label{eq:projection}
{G}_K( u(t,\cdot) )(\xi) := \sum\limits_{k=0}^K \hat{u}_{k}(t) \phi_k(\xi) \; \quad {a.e.} \ \xi \in \Omega.
\end{equation}
The expansion \eqref{eq:gPC_generic} is called generalized polynomial chaos expansion (gPC) \cite{S16}. 
Under mild condition on the probability measure the truncated expansion converges in the sense $|| {G}_K( u(t,\cdot) )(\xi) - u(t,\cdot) || \to \infty$  for $K\to \infty$ as shown in \cite{S2}.
A weak approximation of stochastic systems is obtained substituting the truncated expansion \eqref{eq:projection} into the system itself and projecting the resulting expression onto the subspace of $L^2(\Omega, \mathbb{P})$ spanned by the basis $\{ \phi_k(\xi) \}_{k=0}^K$. 

Moreover  let us recall, as in \cite{FettesPaper,S16} the Galerkin product for any finite $K>0$ and any  ${u}, {z} \in L^2(\Omega,p_\Xi)$,
$\hat{u}=(\hat{u}_i)_{i=0}^K$, $\hat{z}:=(\hat{z}_i)_{i=0}^K$ and for all  $i,j,\ell=0,\dots,K:$
\begin{align*}
\mathcal{G}_K[u, z](t,\xi)
&:=
\sum_{k=0}^{K} ( \hat{u} \ast \hat{z} )_k(t) \phi_k(\xi), \\
(\hat{u} \ast \hat{z})_k(t)
&:=
\sum_{i,j=0}^{K} \hat{u}_i(t) \hat{z}_j(t) \mathcal{M}_\ell, \\
\left( \mathcal{M}_\ell \right)_{i,j} &:= \int_\Omega \phi_i(\xi) \phi_j(\xi)\phi_\ell(\xi) p_\Xi(\xi) d\xi.
\end{align*}
Note that $\mathcal{M}_\ell$ is a symmetric matrix of dimension $(K+1) \times (K+1)$ for any fixed $\ell \in \{0,\dots,K\}.$ 
Moreover, we have  ${\hat{u} \ast \hat{z} = \mathcal{P}(\hat{u}) \hat{z}}$ for $\mathcal{P} \in \mathbb{R}_+^{K+1 \times K+1}$ and $\hat{u} \in \mathbb{R}_+^{K+1}$ defined by
\begin{equation}\label{DefPalpha}
\mathcal{P}(\hat{u})\coloneqq\sum\limits_{\ell=0}^K \hat{u}_\ell\mathcal{M}_\ell.	
\end{equation}
The Galerkin product is symmetric, but not associative ~\cite{S15,S4,S18}.
However, the Galerkin product defined by $\mathcal{G}_K$ is not the only possible projection of the product of random variables $u,z$ on the subspace  $span\{ \phi_0, \dots, \phi_K \},$ nevertheless here we stick to this choice. 

A challenge occurs since only the gPC modes corresponding to the initial data are known. To determine them for $t>0$ we derive a differential equation called stochastic Galerkin formulation, that describes their propagation in time and space. Recovering the stochastic Galerkin formulation at different scales of observation is the aim of the following sections.

\section{Microscopic scale}\label{sec:micro}

Studying individual vehicles and their interactions is the purpose of microscopic models. 
Starting from cars' positions and velocities, the trajectory of vehicles is reconstructed by means of dynamical systems, in the form of ODEs systems. 
We assume that $N$ vehicles are moving along a single-lane infinite road where overtaking is not possible. This means that cars are ordered and have to follow the first one, namely the \textit{leader}. Cars that are not the leader are termed \textit{followers}. We indicate the position $x_i(t)$ and the velocity {$v_i(t)$} of each vehicle $i=1,\dots,N$ at different time $t$.

The most simple dynamics can be described by a first order system of ODEs where the velocity is given as a known function $s(\cdot)$, which depends on the distance between the position of two consecutive vehicles, i.e. $s\left(\frac{L}{x_{i+1}-x_i}\right)$ where $\Delta x_i=x_{i+1}-x_i$ and $L$ is the length of the cars. For simplicity we assume all the vehicles have the same length.

However, one of the main difficulties for drivers might be to estimate their distance from the vehicle in front, which would then influence the velocity itself. Thus, we introduce the random variable $\xi$, as stated in Sec. \ref{sec:unc}, in the initial data as $\Delta x_i (\xi) = x_{i+1} - x_i + \xi$. The uncertainty then affects also the positions, since the velocity is a function of $\Delta x_i$. Therefore we are interested in the evolution of $x_i(t,\xi):\mathbb{R}^+\times \Omega$, for $i=1,\dots,N$:
\begin{align}\label{eq:micro_first_uq}
\begin{cases} 
\dot{x}_i(t,\xi)= {v}_i(t,\xi) \qquad &i=1,\dots N\\
{v}_i(t,\xi) = s\left(\frac{L}{x_{i+1}-x_i+\xi}\right) \quad& i=1,\dots, N-1\\
{v}_N = \bar{s}. &
\end{cases}
\end{align}
Note that the leader, the $N$-th vehicle, has its own dynamics, i.e. an assigned speed value $\bar{s}$. 

Following the method presented in Sec.\ \ref{sec:unc}, the random field $x_i(t,\cdot)$ can be expressed by a spectral expansion 
$x_i(t,\xi)= \sum_{k=0}^\infty \hat{x}_{i_k}(t) \phi_k(\xi)$, where $\hat{x}_{i_k}$ is the $k$-coefficient of the vehicle $i$. The stochastic Galerkin formulation of \eqref{eq:micro_first_uq} reads as:
\begin{align}\label{eq:micro_first_gpc}
\begin{cases} 
\dot{\hat{x}}_{i_k}= \hat{v}_{i_k} \qquad &i=1,\dots N\\
\hat{v}_{i_k} = \widehat{s_{i_k} }\left( \frac{L}{\Delta {x}_{{i}}} \right)  \quad& i=1,\dots, N-1\\
\hat{v}_N = \bar{s}\ e_1 &
\end{cases}
\end{align}

where $\widehat{s_{i_k}}=\int_\Omega s\left( \frac{L}{x_{i+1} - x_i + \xi} \right) \Phi_k (\xi) p(\xi) d \xi $, $e_1=(1,0,\dots,0)^T$, since $\bar{s}$ is a deterministic value. Note that the following approximation holds  when $s$ is linear: $\widehat{s_{i_k}} \left( \frac{L}{\Delta {x}_{{i}} }\right) \approx s\left( \frac{L}{\Delta \hat{x}_{{i}_k}} \right)$, where $\Delta \hat{x}_{{i}_k}=\hat{x}_{{i+1}_k}-\hat{x}_{{i}_k}$.

The system \eqref{eq:micro_first_gpc} is now deterministic with no explicit dependence on the random variable $\xi$. We will show in Sec.\ \ref{sec:micro_to_macro} how we can recover the macroscopic stochastic model starting from \eqref{eq:micro_first_gpc}.

Furthermore, the acceleration term might be also taken into account, which leads to consider second order models:
\begin{align}\label{eq:micro_sec}
\begin{cases}
\dot{x}_i(t,\xi)= {v}_i (t,\xi)\qquad &i=1,\dots N\\
\dot{v}_i(t,\xi) = a(x_{i+1}(t,\xi),x_i(t,\xi),v_{i+1}(t,\xi), v_i(t,\xi)) & i=1,\dots, N-1\\
\dot{v}_N = \bar{a} &
\end{cases}
\end{align}

where $\bar{a}$ describes the dynamics of the leader, independent from the other vehicles. The second equation describes the acceleration which depends on positions and velocities of two consecutive vehicles. Here we consider the formula described in \cite{aw2000SIAP,piu2022stability}, namely $a= C \frac{v_{i+1}(t,\xi)-v_i(t,\xi)}{\Delta x_i^2(t,\xi)}+\frac{A}{t_r}(s(\frac{L}{\Delta x_i(t,\xi)})-v_i(t,\xi))$. In particular, this choice of the acceleration term takes into account the difference with the velocity of the car in front, weighted with the distance between the vehicles, and the attitude to travel with an optimal velocity given by $s$. 

The stochastic Galerkin approach can be applied to \eqref{eq:micro_sec} in the same way. Explicitly, the expanded dynamics of system \eqref{eq:micro_sec} with that particular choice of acceleration function reads as:
\begin{align}\label{eq:micro_sec_st}
\begin{cases}
\sum_{k=0}^\infty \dot{\hat{x}}_{i_k}(t) \phi_k(\xi)=&\sum_{k=0}^\infty \hat{v}_{i_k}(t)\phi_k(\xi)\qquad \\
\sum_{k=0}^\infty \dot{\hat{v}}_{i_k}(t)\phi_k(\xi)=&C \left( \sum_{k=0}^\infty \Delta \dot{\hat{x}}_{i_k}\phi_k(\xi) \right)\ast  \left( \sum_{k=0}^\infty \Delta \hat{x}_{i_k}\phi_k(\xi) \right)^{-2} +\\&\frac{A}{t_r} \left( \widehat{s_{i_k} }  -\sum_{k=0}^\infty \hat{v}_{i_k}(t)\phi_k(\xi)  \right)\\
\dot{v}_N =& \bar{a}  
\end{cases}
\end{align}
where $\Delta \dot{\hat{x}}_{{i}_k}=\hat{v}_{{i+1}_k}-\hat{v}_{{i}_k}$.   
%Let us call $V_{i_k}=s(L\ast \left( \sum_{k=0}^\infty \Delta \hat{x}_{i_k}\phi_k(\xi) \right)^{-1})$ since its explicit computation depends on the choice of the velocity function $s$.
Then, employing the notation introduced above, the projection of \eqref{eq:micro_sec_st} corresponds to
\begin{align}\label{eq:micro_sec_gpc}
\begin{cases}
 \dot{\hat{x}}_{i_k}(t) = \hat{v}_{i_k}(t)\qquad  &i=1,\dots N \\
 \dot{\hat{v}}_{i_k}(t)=C \left( \mathcal{P}^{-2} (\Delta \hat{x}_{i_k}  )\Delta \dot{\hat{x}}_{i_k} \right)   +\frac{A}{t_r} \left( \widehat{s_{i_k} } -\sum_{k=0}^\infty \hat{v}_{i_k}(t)  \right) & i=1,\dots N-1 \\
\dot{v}_N =\bar{a}.  &
\end{cases}
\end{align}

\section{Mesoscopic scale}\label{sec:meso}
Kinetic traffic flow models provide a statistical description of traffic, taking into account both car-to-car interactions and mass distribution of traffic. Such a mixture is particularly suitable for taking uncertain parameters into account. Therefore several works have been done to study the uncertainty at the mesoscopic scale, starting from \cite{wegener1996kinetic} to more recent contributions \cite{tosin2019kinetic,tosin2021boltzmann,tosin2021uncertainty,zanella2020structure}.

Here we consider a kinetic traffic models class of BGK (Bhatnagar, Gross and Krook \cite{bhatnagar1954model}) type, which got recently new insights thanks to the works done in \cite{herty2022uncertainty,herty2020bgk}. In the former, the derivation of a well-posed model linked to the hierarchy of the scales has been provided in the deterministic framework. The latter instead focuses on the stochastic scenario. 

To begin with, we introduce the desired speed $w_i=v_i+h(\rho_i)$, where $h=h(\rho_i):\mathbb{R}^+\to\mathbb{R}^+$ is an increasing, differentiable function of the density called hesitation function \cite{fan2014comparative} which satisfies $h(\rho)\ge0$, $h'(\rho)\ge0$. The variable we will consider in the following is the mass distribution function of the flow, namely $g(t,x,w):\mathbb{R}^+\times\mathbb{R}\times W \to \mathbb{R}^+$, such that $gdxdw$ gives the number of vehicles at time $t\in \mathbb{R}^+$ with position in $[x,x+dx]\subset \mathbb{R}$ and desired speed in $[w,w+dw]\subset W$. Note that macroscopic variables like density and flux function can be recovered from $g$, namely $\rho(t,x)=\int_W g(t,x,w)dw$ and $q=\int_W w g(t,x,w) dw$.

Due to the difficulties in estimating the real distribution of vehicles, we consider a stochastic initial kinetic distribution function $g_0(x,w,\xi)$, where the uncertainty described by the random variable $\xi$ enters in the initial data.

We are interested in the evolution of the random field $g(t,x,w,\xi):\mathbb{R}^+\times\mathbb{R}\times W \times \Omega \to \mathbb{R}^+$ governed by the BGK-kinetic equation:
%\begin{small}
\begin{align}\label{eq:BGK-stoch}
&	\partial_t g(t,x,w,\xi)+\partial_x\Big[ (w-h(\rho(t,x,\xi))) g(t,x,w,\xi)\Big] = \frac{1}{\varepsilon} \Big( M_g(w;\rho(t,x,\xi)) - g(t,x,w,\xi) \Big), \\
&	g(0,x,w,\xi) = g_0(x,w,\xi), 
%&	\rho(t,x,\xi)=\int_W g(t,x,w,\xi) dw.
\end{align}
%\end{small}
where $M_g(w;\rho)$ is the distribution at the equilibrium. 

Applying the intrusive method introduced in Sec.\ \ref{sec:unc}, the random field $g(t,x,w,\xi)$ can be approximated by the spectral expansion truncated at $K$, $g(t,x,w,\xi)=\sum_{k=0}^{K}\widehat{g}_i(t,x,w)\phi_i(\xi),$ and the stochastic Galerkin formulation for \eqref{eq:BGK-stoch} reads as
\begin{align}\label{eq:UBGK}
&	\partial_t \widehat{g}_i(t,x,w) + \partial_x   \Bigl( \Bigl( w Id -   \mathcal{P}\left(h(\widehat{\rho}\left(t,x\right))\right) \Bigr) \widehat{g}(t,x,w) \Bigr)_i
=\frac{1}{\varepsilon} \left( \widehat{M}_i\left(w;\widehat{\rho}(t,x) \right) -  \widehat{g}_i(t,x,w) \right), \\
&	\widehat{g}_i(0,x,w) = \int_\Omega g_{0}(t,x,w,\xi) \phi_i(\xi) {p_\Xi}(\xi) d\xi \qquad \text{for} \, \ i=0,\dots,K
\end{align}
%In the derivation of the previous system \eqref{eq:UBGK} we have used the following  results:\\
%\textcolor{blue}{check if we need these assumptions on the basis functions}

%Under  assumptions (A1)-(A3)  $h$, as given by equation \eqref{h}, fulfills   \cite{gerster2021stability}:
where $\left( \mathcal{P}(h( \widehat{\rho}))  \widehat{g} \right)_i=\sum_{j=0}^K \int_\Omega h\left( \sum_{\ell=0}^K  \widehat{\rho}_\ell \phi_\ell(\xi)
\right)	 \widehat{g}_j \phi_j(\xi) \phi_i(\xi) p_\Xi(\xi) d\xi$ and $w$ is a deterministic variable.

Further, we define for $i=0,\dots,K$

\begin{align}
\widehat{M}_{i}\left(w;\widehat{\rho}(t,x) \right):=&\int_\Omega  M_g\Bigl(w;\sum_{\ell=0}^K \widehat{\rho}_\ell(t,x) \phi_\ell(\xi) \Bigr)\phi_i(\xi)  p_\Xi dw d\xi. 	\label{eq:mg}% \\
%\widetilde{M}_g(w,\widehat{\rho}(t,x),\xi)  :=&  M_g\left( w;\sum_{i=0}^K\widehat{\rho}_i(t,x)\phi_i(\xi)\right).
%\label{eq:mi}
\end{align}

Thus we recovered a stochastic Galerkin system, namely \eqref{eq:UBGK}, for the BGK model.\\
As in \cite{herty2022uncertainty}, it might be employed to detect and forecast regions of high risk of congestions or traffic instabilities. For a detailed analysis and numerical tests we refer to \cite{herty2022uncertainty}.

%Note that, as in the deterministic case, the stochastic density and the flux can be recovered as follow and the Galerkin expansions are straight forward:
%\begin{align}\label{def:rho}
%\rho(t,x,\xi)=\int_W g(t,x,w,\xi)dw = \sum_{i=0}^{\infty} \widehat{\rho}_i\phi_i(\xi), \;  \widehat{\rho}_i=\widehat{\rho}_i(t,x)=\int_W \widehat{g}_i(t,x,w) dw, \\
%q(t,x,\xi)=\int_W wg(t,x,w,\xi)dw = \sum_{i=0}^{\infty}  \widehat{q}_i\phi_i(\xi),  \;  \widehat{q}_i=\widehat{q}_i(t,x)=\int_W w \widehat{g}_i(t,x,w) dw.  \label{def:q}
%\end{align}

\section{Macroscopic scale}\label{sec:macro}
Macroscopic models describe traffic flow in terms of aggregate quantities as density, $\rho=\rho(x,t)$ and mean velocity $v=v(x,t)$ of vehicles at a location $x \in \mathbb{R}$ and time $t>0$. In contrast to the kinetic scale, each reference to the detailed level of vehicles' description is completely lost.

The natural assumption that the total mass is conserved along the road leads to impose that $\rho$ and $v$ satisfy:
$\partial \rho(x,t) + \partial_x \left( \rho(x,t) v(x,t) \right)=0, \ \rho(x,0)=\rho_0(x)$.

However, another relationship has to be provided in order to close the equation, i.e. we have a single equation for two fields. Depending on the closure we distinguish between first and second order models.

In first order models the velocity is given as a function of the density. Among them, one of the most relevant model was introduced by Lighthill, Whitham and Richards (LWR) \cite{lighthill1955RSL,richards1956OR},where a typical choice for the velocity function is $v(x,t)=V(\rho(x,t))=1-\rho$. 

In second order models instead an equation describing the variation of the velocity in time is added to the system. The prototype for second order macroscopic model is given by the ARZ model (Aw, Rascle \cite{aw2000SIAP} and Zhang \cite{zhang2002non}).

However, a challenge occurs when we take into account the uncertainty which affects vehicular traffic.
In the following, we are indeed interested in studying the stochastic scenario, where the uncertainty enters only in the initial data.

In particular, we assume $\rho_0$ to depend on the space and on the random variable $\xi$ introduced in Sec.\ \ref{sec:unc}, i.e. $\rho_0(x,\xi)$, for both first and second order macroscopic models. 
Under the same assumptions as in Sec.\ \ref{sec:unc}, the stochastic LWR model reads as follow:
\begin{align}\label{eq:SLWR}
\partial_t \rho(t,x,\xi) + \partial_x (\rho(t,x,\xi) \ V(\rho(t,x,\xi)))=0\\
\rho(0,x,\xi)=\rho_0(x,\xi).
\end{align}
Then  we follow the same procedure described in Sec. \ref{sec:unc} to get the stochastic Galerkin formulation of the system, which reads as:
\begin{align}\label{eq:SG_LWR}
	\partial_t \widehat{\rho}+\partial_x \left( \mathcal{P}(\widehat{\rho}(t,x)) \widehat{V}_{eq}(\widehat{\rho}(t,x)) \right) =\overrightarrow{0},
\end{align} 
with $\overrightarrow{0}=(0,\dots,0)^T$ vector of $K+1$ components. \\
Note that an arbitrary but consistent gPC expansion $\widehat{V}_{eq}$ is required.
For example, the corresponding gPC expansion of $V_{eq}=1-\rho$ is $\widehat{V}_{eq}=e_1- \widehat{\rho}$, with unit vector $e_1=(1,0,\dots,0)^T$. 

\vspace{0.3cm}

On the other hand, stochastic second order macroscopic models are more challenging to treat. Indeed, beside a more complicated structure, also the properties of the deterministic model have to be preserved, in particular the hyperbolicity of the system has to be ensured.\\
As second order model, we consider the ARZ model as stated before. %which is a system of hyperbolic equations describing the evolution of density and velocity. 
In order to write it in a conservative form, an auxiliary variable is usually introduced, namely $z(t,x)=\rho(v+h(\rho))$. In the deterministic case the system reads as:
\begin{align}\label{eq:ARZ}
&		\partial_t{\rho}+\partial_x(z-\rho h(\rho))=0, \\
&		\partial_t {z}+\partial_x \left(\frac{z^2}{\rho} - z h(\rho)\right)=\frac{1}{\varepsilon}\left(  \left(\rho V_{eq}({\rho}) + \rho h(\rho) \right) -  z\right)\\
& \rho(0,x)=\rho_0(x), \ z(0,x)=z_0(x). \end{align}

As before, we introduce the uncertainty in the initial data, and more precisely in the initial density. This affects also the second equation, which means we end up with a stochastic system depending on $\rho(t,x,\xi), z(t,x,\xi)$.

\begin{remark}
A naive approach would be to substitute the truncated expansion for $\rho$ and $v$ into the stochastic system and then use the Galerkin projection onto the space spanned by the basis functions leads to a loss of hyperbolicity \cite{gerster2021stability}. Indeed, in this case the jacobian of the flux has not necessarily real eigenvalues and a full set of eigenvectors. 
\end{remark}

In order to compute the gPC expansion for the term $\frac{z^2}{\rho}$, the Riemann invariant $w=v+h(\rho)$ is taken into account, such that ${z}=\rho w$, which leads to $\widehat{z}=\mathcal{P}(\widehat{\rho})\widehat{w}$ as in \cite{gerster2021stability}. According to this, the term $\frac{z^2(t,x,\xi)}{\rho(t,x\xi)}=z(t,x,\xi)w(t,x,\xi)$ and the corresponding gPC expansion is $\widehat{z}*\widehat{w}=\mathcal{P}( \widehat{z} ) \mathcal{P}^{-1}(\widehat{\rho}) \widehat{z}$. Thus the stochastic Galerkin formulation for the ARZ model reads as:

\begin{align}
	\partial_t \widehat{\rho} + \partial_x \left( \hat{z}
	-
	\mathcal{P}(\widehat{\rho}) \widehat{h}(\widehat{\rho})\right) = \overrightarrow{0} \\ 
	%-
	\partial_t \widehat{z} + \partial_x \left(
	\mathcal{P}( \widehat{z} ) \mathcal{P}^{-1}(\widehat{\rho}) \widehat{z}
	-
	\mathcal{P}(\widehat{z}) \widehat{h}(\widehat{\rho})\right)=\overrightarrow{0}.
\end{align}

In order to ensure the hyperbolicity of the system as proved in \cite[Thm 2]{gerster2021stability}, the basis functions have to fulfill the following properties
\begin{itemize}
	\item[\textup{(A1)}] \quad
	The  matrices $\mathcal{M}_\ell$ and $\mathcal{M}_k$ commute for all ${\ell,k = 0,\ldots,K}$.
	\item[\textup{(A2)}]\quad
	The matrices~$\mathcal{P}(\widehat{u})$ and $\mathcal{P}(\widehat{z})$ commute for all ${\widehat{u}, \widehat{z} \in \mathbb{R}^{K+1}}$.
	\item[\textup{(A3)}]\quad
	There is an eigenvalue decomposition~${
		\mathcal{P}(\widehat{u}) = V \mathcal{D}(\widehat{u}) V^{T}
	}$
	with constant eigenvectors $V$.	
\end{itemize}
It has been shown that for example the one--dimensional Wiener--Haar basis and piecewise linear multiwavelets fulfill the previous assumptions, but, Legendre and Hermite polynomials do {\em not} fulfill those requirements.

\subsection{From micro to macro}\label{sec:micro_to_macro}
Connections between microscopic and macroscopic traffic flow models are already well established. Indeed there are several works in the literature investigating the limit for first and second order models as \cite{aw2002derivation,di2017many}. The natural question now is if and how the uncertainty influences the relationship between the scales.

\begin{proposition}
Let $\xi$ be a random variable as in Section \ref{sec:micro}, with $N$ cars of fixed length $L$. Assume that $s(\frac{L}{\Delta x})=v(\rho)$.  Then the stochastic ODEs system \eqref{eq:micro_first_uq} converges to the stochastic LWR model \eqref{eq:SLWR} for $L\to 0$ and $N\to \infty$.
\end{proposition}
\begin{proof}

First of all we recall that the uncertainty enters only in the initial data, as in Section \ref{sec:micro}-\ref{sec:macro}.
Then, we define the stochastic local density, according to the deterministic case \cite{aw2002derivation}:
\begin{equation}
	\rho_i^{(N)}(t,\xi)=\frac{L}{\Delta x_i(t,\xi)}=\frac{L}{ x_{i+1}(t,\xi) - x_i(t,\xi)} \qquad i=1,\dots,N-1
\end{equation}
with an abuse of notation, where we assume the uncertainty to be represented by $\xi$ both at the micro and at the macro level. Note that $\rho_i^{(N)}(t,\xi)$ is the same term which appears in the velocity function.

From the definition of local density
\begin{align}
	\frac{d}{dt} \frac{1}{	\rho_i^{(N)}(t,\xi)}&= \frac{1}{L}\left( \frac{d}{dt} x_{i+1}(t)  - \frac{d}{dt} x_{i}(t) +  \frac{d}{dt} \xi \right)\\
	&=\frac{1}{L}\left( s\left(\frac{L}{ x_{i+2}(t) - x_{i+1}(t) + \xi}\right)-s\left(\frac{L}{ x_{i+1}(t) - x_{i}(t) + \xi}\right) \right)\\
	&=\frac{s\left(\rho^{(N)}_{i+1}(t,\xi)\right) - s\left(\rho^{(N)}_{i}(t,\xi)\right)}{L}.
\end{align}

To compute the limit, we follow the procedure as in \cite{burger2018derivation,holden2018follow,tordeux2018traffic}. We consider the Lagrangian coordinate $y$, and we define
\begin{equation}
\rho^{(N)}(y,t,\xi)=\rho_i^{(N)}(t,\xi) \quad \text{if}\ y\in [iL,(i+1)L)
\end{equation}
so we get:
\begin{equation}\label{eq:der_vel}
	\frac{d}{dt} \frac{1}{	\rho^{(N)}(y,t,\xi)}=\frac{s\left(\rho^{(N)}(y+L,t,\xi)\right) - s\left(\rho^{(N)}(y,t,\xi)\right)}{L}.
\end{equation}
By defining 
\begin{equation}
	L=\frac{1}{N}
\end{equation}
and assuming that there exists a function $\rho(y,t,\xi)$ such that $\rho(y,t,\xi)=\lim_{N\to \infty} \rho^{(N)}(y,t,\xi)$, \eqref{eq:der_vel} can be rewritten:
\begin{equation}
\frac{d}{dt} \frac{1}{\rho(y,t,\xi)}= \partial_y s (\rho(y,t,\xi)).
\end{equation}

While $y$ represents the continuous number of cars, we are interested in the position of a car itself, namely $x(y,t)$:
\begin{equation}
	\partial_y x(y,t,\xi)=\lim_{L\to 0} \frac{x(y+L,t,\xi)- x(y,t,\xi)}{L}=\frac{1}{\rho(y,t,\xi)}
\end{equation}
so that
\begin{equation}
x(y,t,\xi)=\int_0^y\frac{1}{\rho(z,t,\xi)}dz,
\end{equation}
assuming $x(0,t,\xi)=0$. 

Then:

\begin{equation}
	\partial_t x(y,t,\xi) = \int_0^y \partial_t \frac{1}{\rho(z,t,\xi)}dz = \int_0^y \partial_z s(\rho(z,t,\xi))dz = s(\rho(y,t,\xi)).
\end{equation}

On the other hand, $y$ can be recovered as:
\begin{equation}
	y(x,t,\xi)=\left( \int_0^x \frac{1}{\rho(z,t,\xi)}dz \right)^{-1}
\end{equation}
and we can compute the derivative with respect to $x$ 
\begin{equation}\label{eq:rosa}
	\partial_x y (x,t,\xi)=\rho(y(x,t,\xi),t,\xi).
\end{equation}

Moreover
\begin{equation}
	\frac{d}{dt} y(x(y,t,\xi),t,\xi) = 0
\end{equation}
since $y(x(y,t+\Delta t,\xi),t+\Delta t,\xi)=y(x(y,t,\xi),t,\xi)$, for $\Delta t \ge 0$.\\
Explicitly
\begin{align}
	\frac{d}{dt} y(x(&y,t,\xi),t,\xi) = \partial_x y(x,t,\xi) \partial_t x( y,t,\xi)+ \partial_t y(x,t,xi) =0\\
	& \partial_t y(x,t,xi) =-s(\rho(y,t,\xi)) \rho(y(x,t,\xi),t,\xi). \label{eq:aran}
\end{align}

We define $\rho(x,t,\xi)=\rho(y(x,t,\xi),t,\xi)$ and we compute the time derivative exploiting \eqref{eq:rosa}-\eqref{eq:aran}
\begin{align}
	\partial_t \rho(x,t,\xi)=& \frac{d}{dt} \rho(y(x,t,\xi),t,\xi)= \\
	=&\partial_y  \rho(y(x,t,\xi),t,\xi)\partial_t y(x,t,\xi)+ \partial_t \rho(y(x,t,\xi),t,\xi)\\
	=&-s(\rho(x,t,\xi))\partial_x \rho(x,t,\xi)+\rho(x,t,\xi)\partial_x s (\rho(x,t,\xi))\\
	=& -\partial_x \left( s(\rho(x,t,\xi)) \rho(x,t,\xi)  \right).
\end{align}
Thus we recover the stochastic LWR model \eqref{eq:SLWR}.
\end{proof}

For second order models, a similar procedure can be applied.

\subsection{From meso to macro}\label{sec:meso_to_macro}

Kinetic and macroscopic scales are closely related, indeed one can recover the macroscopic quantities from the kinetic models, as stated in Sec.\ \ref{sec:meso}. %Furthermore, we consider here the BGK model, where the interaction kernel, typical of kinetic models, is described by a relaxation towards the equilibrium.
Here we are interested in investigating the connections between the models described in Sec.\ \ref{sec:meso} and Sec.\ \ref{sec:macro}.
{We recall that we consider a particular kinetic model, the BGK model, where the interaction kernel is given by a linear term which described a relaxation towards the equilibrium, namely the right hand side term of} \eqref{eq:BGK-stoch}.
	
{In order to understand the intrinsic connection to the macroscopic scale}, let us consider traffic flow at equilibrium conditions. This means that at the kinetic level $M_g=g$ and the desired velocity is the velocity itself, $w=v$, which implies $h=0$. 
Under these assumptions, integrating \eqref{eq:BGK-stoch} with respect to $w$ we  immediately get the stochastic LWR if $w$ coincides with $V_{eq}$.

As far as concerns the second order macroscopic models, the link between the kinetic BGK model and the ARZ has been proved in the deterministic case in \cite{herty2020bgk}, while in \cite{herty2022uncertainty} the link has been extended also to the stochastic framework.
For completeness, we recall here the main result where a gPC formulation of the fluid model obtained  by the stochastic BGK model \eqref{eq:BGK-stoch} is derived.

Further, this model is compared with the stochastic ARZ model. The theorem shows that under assumption \eqref{eq:ass_thm1} the derived gPC  model is equivalent to the stochastic model of \cite{gerster2021stability}. Therein, it has also been shown that the partial differential equation is hyperbolic.
{
\begin{theorem}\label{thm1} \cite[Thm 2.2]{herty2022uncertainty}
	Let $K>0, \varepsilon>0$. Assume the base functions $\{ \phi_0, \dots, \phi_K \}$ fulfill (A1)--(A3) and assume that 
		\begin{align}
	&     \int_W \widehat{M}_i\left(w;\widehat{\rho}(t,x) \right) \ dw = \widehat{\rho}_i(t,x),	\label{eq:coeff_um1}\tag{UM1}\\
	&     \int_W w \ \widehat{M}_i\left(w;\widehat{\rho}(t,x) \right) dw = \Bigl( \mathcal{P} (V_{eq}(\widehat{\rho}(t,x)))\widehat{\rho}(t,x)+\mathcal{P} (h(\widehat{\rho}(t,x)))\widehat{\rho}(t,x)\Bigr)_i.\label{eq:coeff_um2}\tag{UM2}
	\end{align} 
	Let $\widehat{g}_i$ be a strong solution to \eqref{eq:UBGK} and \eqref{eq:mg} for $i=0,\dots,K.$
	\medskip
	Further, assume that for $i=0,\dots,K$ and $(t,x) \in \mathbb{R}^+\times \mathbb{R}$
	\begin{equation}\label{eq:ass_thm1}
	\int_W w^2 \; \widehat{g}_i(t,x,w) dw= (\mathcal{P}(\widehat{q}(t,x))\mathcal{P}^{-1}(\widehat{\rho}(t,x))\widehat{q}(t,x))_i,
	\end{equation}
	where $(\widehat{\rho},\widehat{q})_i$ are the first and second moment of $\widehat{g}_i$ as in \eqref{def:rho}--\eqref{def:q} and $\mathcal{P}$ is defined by \eqref{DefPalpha}.\\
	Then, the functions $(\widehat{\rho},\widehat{q})$ formally fulfill pointwise in $(t,x)\in\mathbb{R}^+\times \mathbb{R}$ and for all $i=0,\dots,K$ the second--order traffic flow model
	\begin{subequations} \label{eq:UARZ}
		\begin{align}
		&   \partial_t \widehat{\rho}_i(t,x) + \partial_x \left[\widehat{q}_i(t,x)- (\mathcal{P}(\widehat{\rho}(t,x))\widehat{\rho}(t,x))_i\right]=0\\
		&   \partial_t \widehat{q}_i(t,x) + \partial_x \left[ (\mathcal{P}(\widehat{q}(t,x))\mathcal{P}^{-1}(\widehat{\rho}(t,x))\widehat{q}(t,x))_i - (\mathcal{P}(\widehat{\rho}(t,x))\widehat{q}(t,x))_i \right]= \\
		& \quad \frac1\varepsilon \Bigl( \Bigl( \mathcal{P} (V_{eq}(\widehat{\rho}(t,x)))\widehat{\rho}(t,x)+\mathcal{P} (h(\widehat{\rho}(t,x)))\widehat{\rho}(t,x)\Bigr)_i  - \widehat{q}_i(t,x) \Bigr) \\
		& \widehat{\rho}_i(0,x) = \int_W \widehat{g}_{0,i}(t,x,w) dw, \\
		& \widehat{q}_i(0,x) = \int_W w \; \widehat{g}_{0,i}(t,x,w) dw.
		\end{align}\end{subequations}
	The system \eqref{eq:UARZ} is hyperbolic for $\widehat{\rho}_i>0.$\\
	Let the  random fields $(\rho,q)=(\rho,q)(t,x,\xi):\mathbb{R}^+\times\mathbb{R}\times\Omega \to \mathbb{R}^2$  be a pointwise a.e.\ solution with second moments w.r.t.\ to $\xi$ of  the stochastic Aw--Rascle--Zhang system with random initial data:
	\begin{subequations} \label{eq:UARZ2}
		\begin{align}
		&		\partial_t{\rho}+\partial_x({q}-{\rho} h({\rho}))=0, \\
		&		\partial_t {q}+\partial_x \Bigl(\frac{{q^2}}{{\rho}} - {q} h({\rho})\Bigr)=\frac{1}{\varepsilon}\left( \rho V_{eq}({\rho}) + \rho h(\rho) -{q}\right), \\
		&\rho(0,x,\xi)=\rho_0(x,\xi), \;  q(0,x,\xi)=q_0(x,\xi).
		\end{align}\end{subequations}
	Under the previous assumptions on the base functions $\{ \phi_0, \dots, \phi_K\}$ and provided that for all $i=0,\dots,K$
	\begin{small}
		\begin{align}\label{cond1}
		\int_\Omega\rho_0(x,\xi) \phi_i(\xi) p_\Xi d\xi = \int_W\widehat{g}_{0,i}(t,x,w) dw, \;
		\int_\Omega q_0(x,\xi) \phi_i(\xi) {p_\Xi} d\xi = \int_W w \widehat{g}_{0,i}(t,x,w) dw,
		\end{align}
	\end{small}
	we have
	\begin{align}\label{thm:final}
	G_K\left(\rho(t,x,\cdot) \right)(\xi)=\sum\limits_{i=0}^K \widehat{\rho}_i(t,x)\phi_i(\xi) \mbox{ and }  G_K\left(q(t,x,\cdot)\right)(\xi)= \sum\limits_{i=0}^K \widehat{q}_i(t,x)\phi_i(\xi),
	\end{align}
	where $(\widehat{\rho},\widehat{q})$ fulfill equation \eqref{eq:UARZ}.
\end{theorem}}

For the proof we refer to \cite{herty2022uncertainty}.

\subsection{Numerical test}
Numerically, we focus on the macroscopic scale, in order to exploit the strength of the stochastic Galerkin approach.
% we are interested in illustrating how the uncertainty, at the macroscopic scale, affects in time characteristic features of traffic models. 
Indeed, thanks to the stochastic Galerkin formulation, the stochastic quantities can be recovered at each time step in every point of the space grid, solving only once the coefficients system.
In particular, in the following we focus on the fundamental diagram. 
To this aim, we run simulations for the stochastic LWR model and we reconstruct an approximation of the stochastic fundamental diagram from the Galerkin coefficients of the density, namely $\tilde{f}(t,x,\xi)=   \mathcal{P}(\widehat{\rho}(t,x) ) \widehat{V_{eq}}$, where $\widehat{V_{eq}}_i=(e_1)_i - \widehat{\rho}_i$ for $i=0,\dots,K$.

We employ the local Lax-Friederichs scheme to solve the PDE for each coefficient \eqref{eq:SG_LWR}.
The numerical parameters are as follows.  We consider the space interval $x\in [a,b]=[0,2]$ and define the uniform spatial grid of size  $\Delta x=1\cdot 10^{-2}$. Moreover, let $T_f=1$ be the final time of the simulations and $\Delta t$ the time step, which is chosen in such a way that the CFL condition is fulfilled. By $N_t$ we denote the number of the time steps needed to reach $T_f$. The random variable $\xi$ is assumed {to} be uniform distributed on $(0,1)$, i.e., $p_\Xi=1$ and $\Omega=(0,1).$ As basis functions we consider the Haar basis. The prototype of the initial data is a Riemann problem:
\begin{equation}
	\rho_0(x,\xi)=\begin{cases}
	\rho_l\equiv \xi \sim \mathcal{U}(u_1, u_2)   \qquad & x<1\\
	\rho_r & x\ge 1
	\end{cases},
%	\qquad
%	v(x,0,\xi)=\begin{cases}
%	0.2   \qquad & x<1\\
%	0.7 & x\ge 1
%	\end{cases}.
\end{equation}
with $\rho_l,\rho_r,u_1,u_2 \in \mathbb [0,1]$.

Moreover, according to the results presented in \cite{gerster2021stability,herty2022uncertainty} where the numerical convergence with respect to $K$ is studied, we choose $K=15$.

In order to understand how the uncertainty in the initial data affects the shape of the fundamental diagram at $T_f=1$, we consider several initial data, fixing $\rho_l\sim \mathcal{U}(0.75, 0.95)$ and varying $\rho_r \in [0,1]$. We mainly focus on the rarefaction case since the density is more spread, with respect to the shock case, and therefore the reconstruction of the fundamental diagram is more accurate. 
In Fig.\ \ref{fig:cloud} we plot the mean, i.e.\ the 0-coefficient, of the fundamental diagram. One may recognize the typical form of the Greenshields flux function. However, a cloud of points is also present approximately for $\rho \in (0.35,0.83)$, which means that the same density value could lead to different velocities. Moreover, we note that this scatter behavior affects only region where there is the transition between free flow and congested flow. 
Furthermore, for $\rho>0.85$ we still observe a scattering in the flux, right box in Fig.\ \ref{fig:cloud}, while in the free flow there is none, left box in Fig.\ \ref{fig:cloud}.

 It is important to note that the scattered fundamental diagram we observe here, is typical of traffic dynamics \cite{Siebel2005}, and it can be recovered from a first order macroscopic model with uncertainty only in the initial data. A deterministic first order model fails in reproducing such a scattered dynamics. .

\begin{figure}
	\centering
	\includegraphics[scale=0.35]{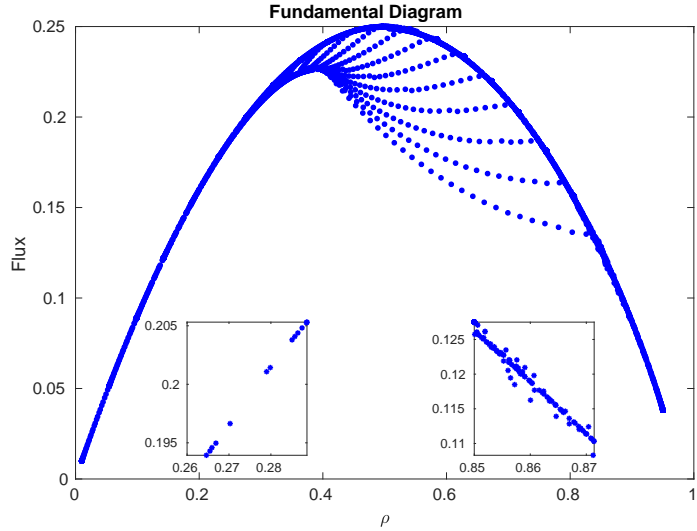}
	\caption{Mean of the fundamental diagram for different initial data}\label{fig:cloud}
\end{figure}

In order to investigate the uncertainty more in details, the variance is also taken into account. The approximation of the variance is computed by the sum of all the coefficients but the first one, squared, i.e.\ for the density $Var(\rho)\approx \sum_{k=1}^K \widehat{\rho}_k^2$. In Fig.\ \ref{fig: FD}, the blue dots stand for the  mean of the reconstructed fundamental diagram while the red bars indicate the mean plus and minus the variance, both for density and flux. It is very interesting to note that while for $\rho\in [0.5,0.8]$ a very small variance in the density corresponds to high variance in the flux, for a more congested traffic the scenario is the opposite, right box in Fig.\ \ref{fig: FD}. This can be explained as follows: in the highly congested traffic situation the velocity is close to $0$ even for some variation of the density values and this causes a very small variation in the flux with respect to the density. 
On the other hand, switching from free flow to congested regime, very small changes in the density values lead to higher uncertainty in the flux value since the velocity has a large impact close to the maximum of the flux, left box in Fig.\ \ref{fig: FD}.

\begin{figure}
	\centering
	\includegraphics[scale=0.4]{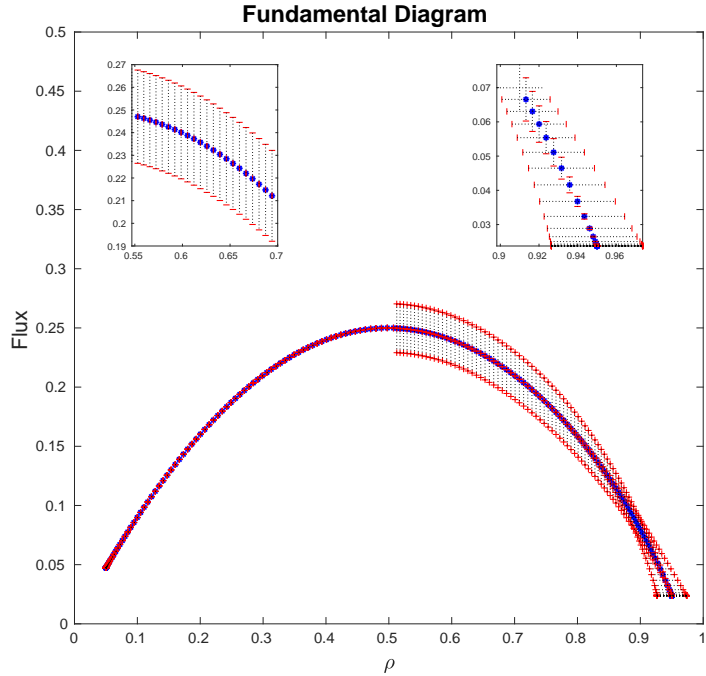}
	\caption{Fundamental diagram with mean and variance.}\label{fig: FD}
\end{figure}

\section{Conclusion and future perspectives}

The presented overview aims at providing a unified framework on how to deal with uncertainty in traffic flow models at different scales of observation. In particular, starting from the microscopic models, to the kinetic equation and finally to the macroscopic ones, the uncertainty was introduced in the initial data in a consistent way. Moreover, the stochastic quantities were treated in the same way following the intrusive stochastic Galerkin approach: first we performed the gPC expansion of the stochastic quantities, we truncated them and put them into the respective evolution equations projecting with the Galerkin ansatz. After presenting the proper stochastic Galerkin formulations for the different systems, we presented some connections between the scales and some numerical simulations which show the intrinsic influence of the uncertainty on the characteristic law of traffic. 

However, a deeper understanding of the link between the Galerkin coefficients at different scales, in particular between microscopic and macroscopic, is still an open question.

\begin{acknowledgement}
The author thanks the Deutsche Forschungsgemeinschaft (DFG, German Research Foundation) for the financial support through 20021702/GRK2326,  333849990/IRTG-2379, HE5386/19-1,22-1,23-1 and under Germany's Excellence Strategy EXC-2023 Internet of Production 390621612.\\
The author is member of the “National Group for Scientific Computation (GNCS-INDAM)”.
\end{acknowledgement}
%
%\section*{Appendix}
%\addcontentsline{toc}{section}{Appendix}

\bibliography{mybib}
\bibliographystyle{spmpsci}
\end{document}